\documentclass[11pt]{article}

\title{Hamilton completion and the path cover number of sparse random graphs}

\author{
Yahav Alon
\thanks{School of Mathematical Sciences, Raymond and Beverly Sackler Faculty of Exact Sciences, Tel Aviv University,
Tel Aviv, 6997801, Israel. Email: yahavalo@tauex.tau.ac.il.}
\and Michael Krivelevich
\thanks{School of Mathematical Sciences, Raymond and Beverly
Sackler Faculty of Exact Sciences, Tel Aviv University, Tel Aviv,
6997801, Israel. Email: krivelev@tauex.tau.ac.il. Partially supported by USA-Israel BSF grant 2018267.}
}

\usepackage{tikz}
\usepackage{mathtools}
\usepackage{verbatim}
\usetikzlibrary{arrows,%
                petri,%
                topaths}%
\usepackage[position=top]{subfig}
\usepackage{amsmath,amsthm, amssymb,latexsym, amsfonts}
\usepackage{algorithm}
\usepackage{enumitem}

\begin{document}
\maketitle
\newtheorem{thm}{Theorem}%[section]
\newtheorem{propos}{Proposition}
\newtheorem{defin}{Definition}
\newtheorem{lemma}{Lemma}[section]
\newtheorem{corol}{Corollary}[section]
\newtheorem{corol*}{Corollary}
\newtheorem{thmtool}{Theorem}[section]
\newtheorem{corollary}[thmtool]{Corollary}
\newtheorem{lem}[thmtool]{Lemma}
\newtheorem{defi}[thmtool]{Definition}
\newtheorem{prop}[thmtool]{Proposition}
\newtheorem{clm}[thmtool]{Claim}
\newtheorem{defini}[thmtool]{Definition}
\newtheorem{conjecture}{Conjecture}
\newtheorem{problem}{Problem}
\newcommand{\Proof}{\noindent{\bf Proof.}\ \ }
\newcommand{\Remarks}{\noindent{\bf Remarks:}\ \ }
\newcommand{\Remark}{\noindent{\bf Remark:}\ \ }

\newcommand{\Dist}[1]{\mathsf{#1}}
\newcommand{\Bin}{\Dist{Bin}}
\newcommand{\HH}{\mathcal{H}}
\newcommand{\pr}{\mathbb{P}}
\newcommand{\sm}{\text{SMALL}}
\newcommand{\cl}{\text{CLOSE}}
\newcommand{\bd}{\text{BAD}}
\newcommand{\lrg}{\text{LARGE}}
\newcommand{\enm}{\text{END}_M}

\begin{abstract}
We prove that for every $\varepsilon > 0$ there is $c_0$ such that if $G\sim G(n,c/n)$, $c\ge c_0$, then with high probability $G$ can be covered by at most $(1+\varepsilon)\cdot \frac{1}{2}ce^{-c} \cdot n$ vertex disjoint paths, which is essentially tight. This is equivalent to showing that, with high probability, at most $(1+\varepsilon)\cdot \frac{1}{2}ce^{-c} \cdot n$ edges can be added to $G$ to create a Hamiltonian graph.
\end{abstract}

\section{Introduction} \label{sec-intro} 

Consider the binomial random graph model $G(n,p)$, in which every edge of the complete graph $K_n$ is retained with probability $p$, independently of all other edges. A \textit{Hamilton cycle} in a graph $G$ is a simple cycle in $G$ that covers the entire vertex set of $G$, and a \textit{Hamiltonian graph} is a graph that contains a Hamilton cycle. A classical result by Koml\'{o}s and Szemer\'{e}di \cite{KS83}, and independently by Bollob\'{a}s \cite{B84}, states that if $p=p(n)=(\log n + \log \log n + f(n))/n$ then
\begin{eqnarray*}
\lim _{n\to \infty} \pr (G(n,p)\text{ is Hamiltonian}) =
	\begin{cases}
        1 & \text{if } f(n) \to \infty ;\\
        0 & \text{if } f(n) \to -\infty .
    \end{cases}
\end{eqnarray*}

In light of this, one cannot expect a Hamilton cycle to exist in a random graph $G(n,p)$ when $p$ is below the sharp threshold stated above. This invites a quantitative question: typically, how close is a random graph below the threshold to being Hamiltonian? To answer this question a measure of ``distance" from Hamiltonicity should first be defined.

One example of such a measure is the maximum length of a cycle in the graph $L_{\max}(G)$, where if $p$ is above the Hamiltonicity threshold then typically $L_{\max}(G(n,p)) = n$. Frieze \cite{F86} showed that if $c>0$ is a large enough constant with respect to $\varepsilon$ then typically
$$
L_{\max}(G(n,c/n)) \ge (1-(1+\varepsilon )ce^{-c})n.
$$
Up to the value of $\varepsilon$, this is a tight bound. This is due to the fact that, with high probability, there are $(c+1+o(1))\cdot e^{-c}n$ vertices of degree less than 2 in $G(n,c/n)$, and such vertices cannot be a part of any cycle. Following that, Anastos and Frieze \cite{AF21} extended this result by showing that, in fact, there is a function $f$ such that $L_{\max}(G(n,c/n))/n \xrightarrow{a.s.} f(c)$. The function $f(c)$ can be expressed as a series, with explicitly computable terms, the first few of which are computed in the paper.

In a recent paper Nenadov, Sudakov and Wagner \cite{NSW} presented a general measure for the distance of a graph from a given property: deficiency. The \textit{deficiency} of a graph $G$ with respect to a graph property $\mathcal{P}$, denoted $\text{def}(G,\mathcal{P} )$, is equal to the smallest integer $t$ such that $G * K_t$ has the property $\mathcal{P}$. Here, $G*H$ is the \emph{join} of $H$ and $G$, that is, the graph with vertex set $V(G) \uplus V(H)$ and edge set $E(G)\cup E(H) \cup (G\times H)$.

With respect to the property of Hamiltonicity, the notion of deficiency is equivalent to other natural measure of distance from Hamiltonicity: it is also equal to the distance in edges of a graph from the nearest Hamiltonian graph. In other words, the deficiency of $G$ with respect to Hamiltonicity is equal to the smallest integer $t$ such that there is a graph $H$ with $e(H)=t$ and $G\cup H$ is Hamiltonian. Evidently, this is also equal to the \emph{path cover number} of $G$, defined as follows.

\begin{defin}
Let $G$ be a graph. The \emph{path cover number} of $G$, denoted by $\mu (G)$, is 0 if $G$ is Hamiltonian, and otherwise it is the minimal number $k\in \mathbb N$ such that there is a covering of $V(G)$ by $k$ vertex disjoint paths.
\end{defin}

Henceforth we will denote $\text{def}(G,\HH )$, where $\HH$ is the property of Hamiltonicity, as $\mu (G)$, as the two quantities are equal.

In the above mentioned paper Nenadov et al. presented a (tight) combinatorial result by presenting an inequality involving $|V(G)|,|E(G)|,\mu (G)$ and showing that it holds for every graph $G$, as well as finding examples of $G$ of every size for which the inequality is an equality.

When random graphs are considered, an immediate bound on $\mu (G)$, similar to the bound on the maximum cycle length, can be derived from the fact that every path in a disjoint path covering of $G$ contains at most two vertices of degree 1 or a single vertex of degree 0. From this we get that, with high probability for $G\sim G(n,c/n)$, one has $\mu (G) \ge (\frac{1}{2}c+1+o(1))\cdot e^{-c}n$. Like in the case of the maximum cycle length, this trivial bound turns out to be tight, as shown in the following theorem, which is our result of this paper.

\begin{thm}\label{main}
For every $\varepsilon > 0$ there is a positive $c_0=c_0(\varepsilon )$ such that, if $G\sim G(n,c/n)$, $c\ge c_0$, then with high probability $$(1-\varepsilon)\cdot \frac{1}{2}ce^{-c} \cdot n \le \mu (G) \le (1+\varepsilon)\cdot \frac{1}{2}ce^{-c} \cdot n.$$
\end{thm}

In Section \ref{sec-per} we introduce notation and definitions to be used later in the paper, as well as some auxiliary results required for our proof. In Section \ref{sec:proof} we present our proof for Theorem \ref{main}.

\section{Preliminaries} \label{sec-per}

Throughout the paper we use the following graph theoretic notation.
For a graph $G=(V,E)$ and two disjoint vertex subsets $U,W\subseteq V$, $E_G(U,W)$ denotes the set of edges of $G$ adjacent to exactly one vertex from $U$ and one vertex from $V$, and $e_G(U,W)=|E_G(U,W)|$.
Similarly, $E_G(U)$ denotes the set of edges spanned by a subset $U$ of $V$, and $e_G(U)$ its size.
The (external) neighbourhood of a vertex subset $U$, denoted by $N_G(U)$, is the set of vertices in $V\setminus U$ adjacent to a vertex of $U$.
The degree of a vertex $v\in V$, denoted by $d_G(v)$, is the number of edges of $G$ incident to $v$, and $d_G(v,U)$ is the number of vertices in $U$ adjacent to $v$. $\Delta (G)$ is the maximum degree of a vertex in $G$, that is, $\Delta (G) =\max _{v\in V}d_G(v)$.

For $u,v\in G$ we let $\text{dist}_G(u,v)$ denote the distance in $G$ between $u$ and $v$, that is the length of a shortest path in $G$ connecting them (or $\text{dist}_G(u,v)=\infty$ if $u$ and $v$ are not connected by any path), where $\text{dist}_G(v,v)$ is defined to be the minimal length of a cycle containing $v$ (or $\text{dist}_G(v,v) = \infty$ if $v$ is not contained in any cycle). 

While using the above notation we occasionally omit $G$ if the identity of the specific graph is clear from the context.

We suppress the rounding signs to simplify the presentation.

The following auxiliary results are required for our proof.

\begin{lemma}\label{binom-coeff}
{\em (Bounds on binomial coefficients, see e.g. \cite{FK}, Chapter 21.1)}
The following inequalities hold for every $n\ge a,k \ge b$.

\begin{enumerate}
\item $\binom{n}{k} \le \left( \frac{en}{k} \right) ^k$\,;\label{binom-coeff-1}
\item $\frac{\binom{n-a}{k-b}}{\binom{n}{k}} \le \left( \frac{k}{n} \right) ^b \cdot \left( \frac{n-k}{n-b} \right) ^{a-b}$\,.\label{binom-coeff-2}
\end{enumerate}
\end{lemma}

\begin{lemma}\label{binom-rv}
{\em (Bounds on binomial random variables, a corollary of Lemma \ref{binom-coeff})}
Let $1\le k \le n$ be integers, $0 < p < 1$, and let $X\sim \Bin (n,p)$. Then the following inequalities hold:

\begin{enumerate}
\item $\pr (X \ge k) \le \left( \frac{enp}{k} \right) ^k$\,;\label{binom-rv-1}
\item $\pr (X = k) \le \left( \frac{enp}{k(1-p)} \right) ^k\cdot e^{-np}$\,.\label{binom-rv-2}
\end{enumerate}

\noindent If, additionally, $k\le np$, then 

\begin{enumerate}[resume]
\item $\pr (X \le k) \le (k+1)\cdot \left( \frac{enp}{k(1-p)} \right) ^k\cdot e^{-np}$\,.\label{binom-rv-3}
\end{enumerate}

\end{lemma}

\begin{lemma}\label{chernoff}
{\em (Chernoff bound for binomial lower tail, see e.g. \cite{CHER})} Let $X\sim \Bin (n,p)$ and $\delta >0$. Then
$$
\pr (X <(1-\delta )np) \le \exp \left( -\frac{\delta ^2 np}{2} \right) .
$$
\end{lemma}

\begin{lemma} \label{lemma:const-degree-vrtcs} {\em (see e.g. \cite{FK}, Theorem 3.3)}
Let $d\in \mathbb{N}$, $c>0$ and let $G\sim G(n,c/n)$. Then, with high probability,
$$
\left| |\{ v\in G \mid d(v) = d \} | - e^{-c} \cdot \frac{c^d}{d!} \cdot n \right| = o\left( n^{0.6} \right).
$$
\end{lemma}

\begin{defini}\label{def:2core}
The \emph{connected 2-core} of a graph $G$, denoted $C_2(G)$, is the 2-core of the largest connected component of $G$. That is, the graph $C_2(G)$ is the maximum size subgraph of the largest connected component in which all vertices have degree at least 2. If $G$ does not have a unique largest component then $C_2(G)=\emptyset$.
\end{defini}

\begin{lemma} \label{lemma:sizeof2core} {\em (see e.g. \cite{FK}, Lemma 2.16)}
Let $\varepsilon > 0$. Then there is $C=C(\varepsilon )>1$ such that, if $c\ge C$, $x<1$ is the unique solution  in $(0,1)$ to $xe^{-x}=ce^{-c}$, and $G\sim G(n,c/n)$, then, with high probability,
$$
|V(C_2(G))| = (1-x)\cdot \left( 1-\frac{x}{c} +o(1) \right) \cdot n.
$$

Additionally, the solution $x$ satisfies $x=ce^{-c}+c^2e^{-2c}+o(e^{-2c})$. Therefore, with high probability,
$$
|V(C_2(G))| \ge (1-(c+1)e^{-c}-(1+\varepsilon )c^2e^{-2c})n.
$$
\end{lemma}

We will also use a strong variant of the Azuma-Hoeffding inequality due to Warnke. To introduce it, we first need to define a vertex exposure martingale.

\begin{defini}(vertex exposure martingale, see \cite{ALONSPENCER}, Section 7.1)
Let $f$ be a graph theoretic function, and let $G'\sim G(n,p)$, where $V(G')$ is assumed to be $[n]$, and $G$ is a graph on $[n]$. The \emph{vertex exposure martingale} corresponding to $G$ and $f$ is the sequence of random variables $X_0,...,X_n$, where
$$
X_i \coloneqq \mathbb E \left[ f(G') \mid \forall x,y\in [i]:\{x,y\}\in E(G') \leftrightarrow \{x,y\}\in E(G) \right].
$$
\end{defini}

\begin{lemma} \label{lemma:martingales} {\em (an immediate consequence of \cite{WAR}, Theorem 1.2)}
Let $G\sim G(n,p)$, let $f$ be a graph theoretic function, and let $X_0,X_1,...,X_n$ be the corresponding vertex exposure martingale. Further assume that there is a graph property $\mathcal{P}$ and a positive integer $d$ such that, for every two graphs $G_1,G_2$ on $V$ such that $E(G_1)\triangle E(G_2) \subseteq \{v\}\times V$ (the symmetric difference between $E(G_1)$ and $E(G_2)$) for some $v\in V$, the following holds:
\begin{eqnarray*}
|f(G_1)-f(G_2) | \le
	\begin{cases}
        d & \text{if } G_1,G_2\in \mathcal{P} ;\\
        n & \text{otherwise} .
    \end{cases}
\end{eqnarray*}
Then
$$
\pr \left( (X_n \ge X_0 + t )\ \wedge \ (G\in \mathcal{P}) \right) \le \exp \left( -\frac{t^2}{2n(d+1)^2} \right) .
$$

\end{lemma}

\section{Proof of Theorem \ref{main}} \label{sec:proof}

In this section we present a proof of Theorem \ref{main}. Recall that only the upper bound in the statement requires a proof. Throughout the proof we assume that $c$ is large enough (with respect to $\varepsilon$) without stating this explicitly.

\textbf{Proof outline:} Let $V_1\subseteq V(G)$ be the set of vertices of degree 1. We prove the claim by showing that for a large enough value of $c$, with high probability, if we add to $G$ the edges of a (partial) matching on $V_1$, the resulting graph contains a cycle of length at least $(1-\frac{1}{2}\varepsilon ce^{-c})n$. Since, with high probability, $|V_1|\le (1+\frac{1}{2}\varepsilon)ce^{-c}n$, by removing the added matching from the cycle we get a set of at most $(1+\frac{1}{2}\varepsilon)\frac{1}{2}ce^{-c}n$ paths that cover at least $(1-\frac{1}{2}\varepsilon ce^{-c})n$ vertices. The remaining vertices can now be covered by paths of length 0 to establish the result.

In Section \ref{sec:cores} we study the structure of $G$, and show its typical properties that will help with proving the likely existence of a long cycle. We then describe a construction of an auxiliary graph $G^*$ and a matching $M\subseteq G^*$ whose vertex set is contained in $N_G(V_1)$, such that if $G^*$ contains a Hamilton cycle that contains all the edges of $M$ then $G$ along with the matching on $V_1$, obtained from $M$ by translating an edge between two vertices of $N_G(V_1)$ to an edge between their neighbours in $V_1$, contains a sufficiently long cycle. In short, $G^*$ is obtained by adding $M$, a matching on $N_G(V_1)$, to the subgraph induced by $G$ on a subset of the vertices. The vertex set of $G^*$ is obtained by removing vertices from $V(G)$ that may prove problematic, such as pairs of vertices with very small degree that are within a short distance from one another. This will enable us to claim that, with high probability, the obtained $G^*$ has some good expansion properties, which will be useful in the next section.

Finally, in Section \ref{sec:hamiltonian} we prove that indeed, with high probability, $G^*$ contains such a Hamilton cycle. Here, we will employ adapted variations on methods used in the setting of finding Hamilton cycles in expanding graphs, namely, rotations and extensions. For an overview on this subject we refer the reader to \cite{SURV}, Section 2.3.

\subsection{Properties of $G$ and its 2-core} \label{sec:cores}

Recall that $V_1$ denotes the set of vertices in $G$ of degree 1. Define the following three subsets of $V(G)$.

\begin{eqnarray*}
\sm &=& \sm (G) = \{ v\in V(G) \mid d_G(v,V(G)\setminus N(V_1)) < c/1000 \}; \\
\lrg &=& \lrg (G) = \{ v\in V(G) \mid d_G(v) > 20c \}; \\
\cl &=& \cl (G) = \{ v\in \sm \mid \exists u\in \sm : \text{dist}_G(v,u) \le 4 \} .
\end{eqnarray*}
Notice that $V_1 \subseteq \sm$, and that $\cl$, among other elements, contains all the  vertices of $\sm$ that belong to cycles of length 3 or 4. Informally, vertices of $\sm$ and $\cl$ can potentially be problematic later in the proof, when $G^*$ is constructed, and therefore will require special treatment. In light of this, it will be useful to show that these sets are typically not very big. This, among other typical properties of $G$, is proved in the following lemma.

\begin{lemma} \label{lemma:properties}
With high probability $G$ has the following properties:

\begin{enumerate}[{label=\textbf{(P\arabic*)}}]
%    \item
%      \label{item:maxdegree}
%      $\Delta (G) \le \log n$;

    \item
      \label{item:V1-is-small}
      $(1-n^{-0.4})\cdot ce^{-c}\cdot n \le |V_1| \le (1+n^{-0.4})\cdot ce^{-c}\cdot n$;
    \item
      \label{item:SMALL-is-small}
      $|\sm | \le e^{-0.9c}\cdot n$;
    \item
      \label{item:BIG-is-small}
      $|\lrg | \le 10^{-6}\cdot n$;
    \item
      \label{item:CLOSE-is-small}
      $|\cl | \le e^{-1.8c}\cdot n$;
    \item
      \label{item:not-too-many-edges}
      every $U\subseteq V(G)$ with $|U| \le 10^{-5}\cdot n$ has $e(U) < 10^{-4}c\cdot |U|$;
    \item
      \label{item:not-too-few-edges}
      every $U,W\subseteq V(G)$ disjoint with $|U| = 10^{-6}\cdot n$ and $|W|=\frac{1}{5}n$ satisfy: $e(U,W) \ge 10^{-7} c\cdot n$.

%    \item
%      \label{item:not-too-many-edges}
%      every $U\subseteq V(G)$ with $|U| \ge e^{-c}\cdot n$ has $e(U)+e(U,V(G)\setminus U) \le 4c\cdot |U|$;
%    \item
%      \label{item:not-few-many-edges}
%      if $U\subseteq V(G)$ and $\frac{1}{16}n\le |U| \le \frac{1}{2}n$ then $e(U) \ge \frac{c}{24}\cdot n$;
%    \item
%      \label{item:expansion}
%      if $U\subseteq V(G) \setminus \sm (G)$ and $|U| \le \frac{1}{6}\cdot n$ then $|N_G(U)|\ge 4|U|$.  
      
\end{enumerate}

\end{lemma}

\begin{proof}

For each of the given properties, we bound the probability that $G\sim G(n,p)$ fails to uphold it.

\begin{itemize}

	\item[\ref{item:V1-is-small}.]
	This is an immediate consequence of Lemma \ref{lemma:const-degree-vrtcs}, with $d=1$.

    \item[\ref{item:SMALL-is-small}.]
      Since $|N(V_1)|\le |V_1|$, assuming $G$ has \ref{item:V1-is-small}, the probability that $|\sm | \ge s \coloneqq e^{-0.9c}n$ is at most the probability that there is a set $U$ of size $s $ and a set $W$ of size $t=2 ce^{-c}\cdot n$ with less than $\frac{c}{1000}\cdot s$ edges between $U$ and $V(G)\setminus (U\cup W)$. By Lemma \ref{binom-coeff}(\ref{binom-coeff-1}) and Lemma \ref{binom-rv}(\ref{binom-rv-2}), the probability for this is at most
      
	\begin{eqnarray*}
	& & \binom{n}{s}\cdot \binom{n}{t}\cdot \pr \left( \Bin \left( s(n-s-t), p \right) < \frac{cs}{1000} \right) \\
	& \le & \left(\frac{en}{s} \right) ^{s} \cdot \left(\frac{en}{t} \right) ^{t} \cdot \frac{cs}{1000}\cdot \pr \left( \Bin \left( s(n-s-t), p \right) = \frac{c}{1000} \cdot s \right) \\
	& \le & O(n)\cdot e^{0.91c\cdot s} \cdot e^{2c\cdot t} \left( \frac{1000es(n-s-t)p}{cs(1-p)} \right) ^{10^{-3}cs}\cdot e^{-s(n-s-t)p} \\
	& \le & e^{0.92c\cdot s}  \cdot (1001e) ^{10^{-3}cs}\cdot e^{-0.99c\cdot s} \\
	& \le & e^{-0.06cs} \\
	& = & o(1).
	\end{eqnarray*}

    \item[\ref{item:BIG-is-small}.]
      If there are more than $10^{-6} n$ vertices with degree more than $20c$, then in particular there is a set $U\subseteq V(G)$ of size $10^{-6}n$ with $e(U)+e(U,V(G)\setminus U) \ge 10^{-5}cn$. By Lemma \ref{binom-coeff}(\ref{binom-coeff-1}) and Lemma \ref{binom-rv}(\ref{binom-rv-1}), the probability of this is at most
      
	\begin{eqnarray*}
	& & \binom{n}{10^{-6}n}\cdot \pr \left( \Bin \left( \binom{10^{-6}n}{2}+10^{-6}\cdot(1-10^{-6})n^2, p \right) \ge 10^{-5}c n \right) \\
	& \le & 2^n \cdot \left( \frac{e\cdot 10^{-6}n^2\cdot p}{10^{-5}c n} \right) ^{10^{-5}c n} \\
	& \le & 2^n \cdot \left( \frac{e}{10} \right) ^{10^{-5}c n} \\
	& = & o(1).
	\end{eqnarray*}

   \item[\ref{item:CLOSE-is-small}.]   
    In order to bound the probability that $G$ fails to satisfy \ref{item:CLOSE-is-small} we aim to apply Lemma \ref{lemma:martingales} with $f(G) = |\cl (G)|$ and $\mathcal{P}$ the property $\Delta (G) \le \log n$. First, we bound $X_0 = \mathbb{E}\left[ f(G) \right]$ from above.
   
   Let $Y_1$ denote the random variable counting cycles in $G$ of length 3 or 4, and $Y_2$ the random variable counting the number of paths in $G$ of length at most 4, such that both their endpoints are in $\sm$. Then, deterministically, $|\cl| \le Y_1 + 2Y_2$, and therefore $\mathbb{E}[\cl] \le \mathbb{E}[Y_1] + 2\mathbb{E}[Y_2]$. It is already known that $\mathbb{E}[Y_1]=O(1)$ (see e.g. \cite{FK}, Theorem 5.4), so it remains to bound $\mathbb{E}[Y_2]$.
   
   By the definition of $\sm$, if $v\in \sm$ then either $v$ has a small total degree in $G$, say, at most $10^{-3}c + \delta$, or $v$ has more than $\delta$ neighbour in $N(V_1)$. We utilize this fact for $\delta=10^{-4}c$ to bound the expectation of $Y_2$.
   
    For a given pair of vertices $u,v\in V(G)$ and a path $P$ of length $\ell \le 4$ between them, the probability that $P\subseteq G$ and $u,v\in \sm$ is at most $p^{\ell}$ times the probability of the event: at least one of the vertices $u,v$ has at least $10^{-4}c$ neighbours in $V(G)\setminus V(P)$ such that each of these neighbours has a neighbour of degree 1, or $e_G(\{u,v\},V(G)\setminus V(P)) \le 2(10^{-3} + 10^{-4})c)$. The probability of this event is at most
   \begin{eqnarray*}
   & & 2\cdot (n^2p^2)^{c/10000}\cdot (1-p)^{c(n-5)/10000} + \pr \left( \Bin \left( 2(n-\ell -1),p \right) \le 2(10^{-3} + 10^{-4})c \right) \\
   & \le & e^{-c^2/20000} + c\cdot \left( 2000e \right) ^{c/450}\cdot e^{-2c} \\
   & \le & e^{-1.95c}.
   \end{eqnarray*} 
   From this we get
   $$
   \mathbb{E}[Y_2] \le \sum _{\ell =1}^4 n^{\ell +1}p^{\ell}e^{-1.95c} \le c^4e^{-1.9c}n,
   $$
   and therefore $\mathbb{E}[|\cl|] \le \mathbb{E}[Y_1] + 2\mathbb{E}[Y_2] \le e^{-1.85c}n$.
   
   In order to apply Lemma \ref{lemma:martingales} it remains to determine the parameter $d$. We claim that setting $d = 3\log ^7 n$ satisfies the conditions of the lemma. Let $G_1,G_2$ be two graphs on $V$ with maximum degree at most $\log n$, that differ from each other only in edges incident to $v\in V$. It suffices to show that $\cl(G_1) \triangle \cl(G_2)$ only contains vertices of distance at most 7 from $v$ in either $G_1$ or $G_2$. Indeed, let $u$ be (without loss of generality) in $\cl(G_1) \setminus \cl(G_2)$. This can occur for two reasons.
   
   \begin{enumerate}
   \item $v$ is part of a path of length at most 4 in $G_1$ between $u$ and another vertex, or a cycle of length at most 4 that contains $u$. This can only occur if $\text{dist}_{G_1}(u,v) \le 3$.
   \item $u$ or a vertex of distance at most 4 from $u$ are in $\sm (G_1) \setminus \sm (G_2)$. Call this vertex $w$ (possibly $w=u$). Going from $G_2$ to $G_1$, this can happen because $\{v,w\} \in E(G_2)$, or because changing the edges of $v$ caused another neighbour of $w$ to be moved into $N_{G_1}(V_1(G_1))$. In any of these cases the distance between $v$ and $w$ is at most 3 in one of the graphs $G_1$ or $G_2$, and therefore the distance between $u$ and $v$ is at most 7 in that graph.
   \end{enumerate}
   
   Now we get by Lemma \ref{lemma:martingales}
   \begin{eqnarray*}
   & & \pr \left( |\cl| > \mathbb{E}[|\cl|] + n^{2/3} \right) \\
   & \le & \pr \left( |\cl| > \mathbb{E}[|\cl|] + n^{2/3} \ \wedge \ \Delta (G) \le \log n \right) + \pr \left( \Delta (G) > \log n \right) \\
   & \le & \exp \left( -\frac{n^{4/3}}{2n(3\log ^7n+1)^2} \right) + n\cdot \binom{n}{
   \log n}\cdot \left( \frac{c}{n} \right) ^{\log n} \\
   & = & o(1),
   \end{eqnarray*}    
    which implies that $|\cl | \le e^{-1.8c}n$ with high probability.
      
    \item[\ref{item:not-too-many-edges}.]
      By Lemma \ref{binom-coeff}(\ref{binom-coeff-1}) and Lemma \ref{binom-rv}(\ref{binom-rv-1}), the probability that there is a set $U\subseteq V(G)$ of size $k\le 10^{-5}n$ that contradicts \ref{item:not-too-many-edges} is at most
      
      $$
       \binom{n}{k}\cdot \pr \left( \Bin \left( \binom{k}{2},p \right) \ge 10^{-4}ck \right) 
       \le \left( \frac{en}{k} \right) ^k \cdot \left( \frac{10^4ek^2p}{2ck} \right) ^{10^{-4}ck} 
       \le \left( \frac{10^4ek}{2n} \right) ^{10^{-5}ck} .
      $$
     
      Observe that if $k\le \log n$ this expression is of order $n^{-\Omega (1)}$. By the union bound the probability that $G$ does not have \ref{item:not-too-many-edges} is at most
      $$
      \sum _{k=1}^{10^{-5}n} \left( \frac{10^4ek}{2n} \right) ^{10^{-5}ck}
      \le \log n \cdot n^{-\Omega (1)} + \sum _{k=\log n}^{10^{-5}n} \left( \frac{e}{20} \right) ^{10^{-5}ck}
      = o(1).
      $$
      
    \item[\ref{item:not-too-few-edges}.]
      Since there are at most $3^n$ ways to choose two disjoint subsets of $V(G)$, by the union bound and by Lemma \ref{chernoff} with parameters $m=2\cdot 10^{-7}\cdot n^2$ and $\delta = \frac{1}{2}$, the probability that there are disjoint $U,W\subseteq V(G)$ with $|U| = 10^{-6}\cdot n$ and $|W|=\frac{1}{5}n$ and less than $10^{-7} c\cdot n$ edges between them is at most $3^n\cdot e^{-\frac{1}{4}\cdot 10^{-7}cn} = o(1).$

\end{itemize}

\end{proof}

Towards constructing $G^*$, following is the definition of the set $\bd = \bd (G)\subseteq V(G)$, whose vertices we would like to exclude from $G^*$. Let $X\subseteq V(G)$ be a minimum size set such that $d_G(v,\sm \cup X) \le 1$ for every $v\notin X$, let $Y \coloneqq \{ v\in V(G) \mid d_G(v) =2, d_G(v,X) = 1 \}$ and set
$$
\bd \coloneqq X \cup Y.
$$
Informally, excluding $\bd$ from our constructed graph $G^*$ (rather than just excluding $\sm$) will ensure that the remaining vertices in the graph all have very few neighbours in $G$ outside of $V(G^*)$.

Observe that $X$, and therefore $\bd $, are well defined: if $X_1,X_2 \subseteq V(G)$ satisfy $\forall v\notin X_i: d_G(v,\sm \cup X_i) \le 1$, $i=1,2$, then $\forall v\notin X_1\cap X_2: d_G(v,\sm \cup (X_1\cap X_2)) \le 1$ also holds, and therefore there is a unique smallest set $X$ that satisfies the condition. 

\begin{lemma} \label{lemma:bad-is-small}
With high probability $| \bd | \le 3\sqrt{c}e^{-c}n$.
\end{lemma}

\begin{proof}
We first show that, with high probability, $|X|\le x\coloneqq \sqrt{c}e^{-c}n$. Let $A$ denote the complementing event, and observe the vertices of $X$ by order of addition to $X$, according to the following process: while there is a vertex outside $X$ with degree at least 2 into $\sm \cup X$, add the first such vertex (according to some pre-determined order). Assume that $A$ happened, that is, assume that $|X|\ge x$, and let $X'$ be the set containing the first $x$ vertices added to $X$. Then there is a subset $Z\subseteq \sm$ such that $|Z| \le 2x$ and $e(Z\cup X') \ge 2x$. By the definition of $\sm$, this means that the degree of every vertex in $Z$ into $V(G)\setminus N(V_1)$ is at most $\frac{c}{1000}$.

Therefore, assuming $G$ has property $\ref{item:V1-is-small}$, $A$ is contained in the following event: there are sets $X',Z,S,T\subseteq V(G)$ such that the following hold:
\begin{enumerate}
\item $|X'| = x,\ |Z|\le 2x,\ |T|=|S|\le 2ce^{-c}n$;
\item $e(Z\cup X') \ge 2x$;
\item $e(Z,V(G)\setminus (X'\cup Z\cup S\cup T)) \le \frac{c}{1000}\cdot |Z| $;
\item all the vertices in $T$ have degree 1, and their unique neighbour is in $S$, so that no two vertices in $T$ share their neighbour.
\end{enumerate}

Here, $T=V_1$ and $S$ is a set of size $|T|$ that contains $N(V_1)$.

For given sets $X',Z,S,T$ of respective sizes $x,z,s,s$, where the set $T$ is ordered, the probability that all the conditions above are satisfied is at most
\begin{equation}\label{eqn:bad}
p^s(1-p)^{s(n-2)}\cdot \pr \left( \Bin \left( \binom{3x}{2},p \right) \ge 2x \right) \cdot \pr \left( \Bin \left( z(n-z-x-2s),p \right) \le 10^{-3}cz \right)
\end{equation} 

We first use Lemma \ref{binom-rv} to bound from above some of the terms in Equation (\ref{eqn:bad}), under the assumption $s \le 2ce^{-c}n$.

First,
$$
\pr \left( \Bin \left( \binom{3x}{2},p \right) \ge 2x \right) \le \left( \frac{9ex^2p}{4x} \right) ^{2x} \le e^{-1.9cx}.
$$
Next, if $s \le 2ce^{-c}n$ then
$$
\pr \left( \Bin \left( z(n-z-x-2s),p \right) \le 10^{-3}cz \right) \le n\cdot \left( \frac{1000ez\cdot np}{cz\cdot (1-p)} \right) ^{10^{-3}cz} \cdot e^{-0.99cz} \le e^{-0.9cz}.
$$
Altogether we get that the expected number of such sets is at most
\begin{eqnarray*}
&  & \binom{n}{s}\cdot (np\cdot e^{-(n-2)p})^s \binom{n}{x} \binom{n}{z} \cdot \pr \left( \left( \binom{3x}{2},p \right) \ge 2x \right) \cdot \pr \left( \Bin \left( z(n-z-x-2s),p \right) \le 10^{-3}cz \right)\\
& \le & \left( \frac{en}{x} \right) ^x \cdot \left( \frac{en}{z} \right) ^z \cdot \exp \left( 2ce^{-c}n-1.9cx-0.9cz \right) \\
& \le & \exp \left( e^{-c}n+2ce^{-c}n-0.7cx \right) \\
& \le & e^{-cx/2};
\end{eqnarray*}
where in the second inequality we used the fact that $\left( \frac{e\cdot e^{-c}n}{z} \right) ^z \le e^{e^{-c}n}$, since the expression is maximized when $z=e^{-c}n$.

Finally, summing over all $O(n^2) = o(e^{cx/2})$ possibilities for $s,z$, by the union bound, we get that the probability that sets $X',Z,S,T\subseteq V(G)$ satisfying the above-listed conditions exist, and therefore $\pr (A)$, is of order $o(1)$.

To finish the proof we observe that every vertex in $Y\setminus X$ has a unique neighbour in $X$, and that for two vertices $u,v\in Y\setminus (\cl \cup X) \subseteq \sm \setminus \cl$ these unique neighbours in $X$ are distinct, since otherwise $\text{dist}(u,v) \le 2$, a contradiction. It follows that $|Y\setminus X|\le |X|+|\cl |$ and therefore, due to property \ref{item:CLOSE-is-small}, with high probability we have 
$$
|\bd |\le |X|+|Y\setminus X| \le 2|X| + |\cl | \le 3\sqrt{c}e^{-c}n .
$$
\end{proof}

Recall Definition \ref{def:2core} and let $C_2=C_2(G)$. Now set
$$
V(G^*) = V(C_2) \setminus \left( \cl \cup \bd \right) ,
$$
set $M$ to be a pairing on the vertex set $N_G(V_1 \setminus \cl ) \cap V(G^*)$ by order, that is, the smallest vertex of $N_G(V_1 \setminus \cl ) \cap V(G^*)$ is paired to the second smallest and so on, so that all vertices are paired except possibly the largest one, and set $E(G^*) = E_G(V(G^*))\cup E(M)$. Observe that no two vertices in $V(M)$ are connected by an edge of $G^*\setminus M$. Indeed, if two such vertices are connected by an edge, then  all their neighbours in $V_1$ must also be in $\cl$, which implies that these vertices are not in $N_G(V_1 \setminus \cl )$, and therefore not in $V(M)$. Also observe that $d_{G^*}(v,V(G^*)\setminus V(M)) \ge 2$ for every $v\in V(G^*)$. Indeed, if $v\in V(G^*)\setminus \sm$ then it cannot have more than one neighbour in $Y \setminus \cl \subseteq \sm \setminus \cl$, since this implies that its (at least) two neighbours in this set belong to $\cl$, a contradiction. Furthermore, it cannot have more than two neighbours in $X \cup \cl \subseteq X \cup \sm$, since every vertex outside of $X$ has at most one neighbour in $X\cup \sm$, and $V(G^*)\cap X = \emptyset$ since $X\subseteq \bd$. Overall, $v$ cannot have more than three neighbour in $(Y\setminus \cl )\cup X\cup \cl = \bd \cup \cl$, and therefore we get
$$
d_{G^*}(v,V(G^*)\setminus V(M)) \ge d_{G}(v,V(G)\setminus V(M)) -3 \ge \frac{c}{1000}-3\ge 2.
$$
On the other hand, if $v\in V(G^*)\cap \sm$ then $v$ has no neighbours in $\cl$, since otherwise it would be in $\cl$ itself. Additionally, if $v$ has degree exactly 2 into $V(G)\setminus V(M)$ then it no neighbours in $\bd$, since otherwise, by definition, it would be a member of $X\cup Y=\bd$. Finally, if $v$ has degree more than 2 into $V(G)\setminus V(M)$ then it has no neighbour in $Y$ since otherwise it is not a member of $\cl$, and at most one neighbour in $X$, and therefore $v$ has at most one neighbour in $\bd$. In each of these cases we conclude that $d_{G^*}(v,V(G^*)\setminus V(M)) \ge 2$.

\begin{lemma} \label{lemma:Gstar-is-big}
With high probability $|V(G^*)| \ge (1-(1+\frac{1}{4}\varepsilon)\cdot ce^{-c})\cdot n$.
\end{lemma}

\begin{proof}
The lemma follows immediately from property \ref{item:CLOSE-is-small} in Lemma \ref{lemma:properties}, from the fact that, with high probability, $|V(C_2)| \ge (1-(1+\frac{1}{8}\varepsilon)\cdot ce^{-c})\cdot n$ due to Lemma \ref{lemma:sizeof2core} with $\frac{1}{8}\varepsilon$, and from Lemma \ref{lemma:bad-is-small}.
\end{proof}

\begin{lemma} \label{lemma:M-is-big}
With high probability $|V(M)| \ge (1-\frac{1}{4}\varepsilon )\cdot ce^{-c}n $.
\end{lemma}

\begin{proof}
Let $V_0 \subseteq V(G)$ be the set of isolated vertices in $G$. By Lemmas \ref{lemma:const-degree-vrtcs} and \ref{lemma:sizeof2core}, with high probability
$$
|V(G)\setminus (V_1 \cup V_0 \cup C_2)| \le 2c^2e^{-2c}n.
$$
By this inequality, and by the high probability bounds on the sizes of $\bd$ and $\cl$ we get from \ref{item:CLOSE-is-small} and Lemma \ref{lemma:bad-is-small}, we get
$$
|V(M)| \ge |V_1| -|\cl |- |\bd | - |V(G)\setminus (V_1 \cup V_0 \cup C_2)| \ge (1-\frac{1}{4}\varepsilon )\cdot ce^{-c}n.
$$
\end{proof}

Recall that $M$ is a (partial) pairing of $N_G(V_1)$. For every $v\in V(M)$ denote by $v'$ the unique neighbour of $v$ in $V_1$, and let $V_1^*\coloneqq \{v' \mid v\in V(M)\}\subseteq V_1 \setminus \cl$. Define a matching $M'$ on $V_1^*$ by matching the pair $v',u'$ for every $\{v,u\}\in E(M)$.

In the next section we show that, with high probability, $G^*$ contains a Hamilton cycle that contains all the edges of $M$, and show that this implies Theorem \ref{main}.

\subsection{$M$-Hamiltonicity of $G^*$} \label{sec:hamiltonian}

Call a path $P\subseteq G^*$ an $M$\emph{-path} if for every edge $e\in M$, either $e\in E(P)$ or $e\cap V(P) = \emptyset$, and similarly define an $M$\emph{-cycle}. Say that a graph containing $M$ is $M$\emph{-Hamiltonian} is it contains a Hamilton $M$-cycle. Note that, by its definition, a Hamilton $M$-cycle must contain all the edges of $M$. For a non-$M$-Hamiltonian subgraph $H\subseteq G$ we say that $e\notin E(H)$ is an $M$\emph{-booster} with respect to $H$ if $H\cup \{e\}$ is $M$-Hamiltonian, or a maximum length $M$-path contained in $H\cup \{e\}$ is strictly longer than a maximum length $M$-path contained in $H$. Finally, say that a graph $\Gamma$ on $V(G^*)$ is an $M$\emph{-expander} if $M\subseteq \Gamma$, and for every $U\subseteq V(G^*)$ with $|U| \le \frac{1}{4}n$ the inequality $\left| N_{\Gamma}(U) \setminus V(M) \right| \ge 2|U|$ holds.

\begin{lemma} \label{lemma:many-boosters}
For every $M$-expander $\Gamma$ on $V(G^*)$ such that $\Gamma$ is not $M$-Hamiltonian, $\binom{V(G^*)}{2}$ contains at least $\frac{1}{32}n^2$ $M$-boosters with respect to $\Gamma$.
\end{lemma}

A similar statement to this is made in \cite{AF21}, Section 3.5. For completeness, and since the authors of the aforementioned paper use some tools different than ours in their proof, we add our somewhat simpler version of a proof to this claim.

\begin{proof}

Let $\Gamma \supseteq M$ be an $M$-expander with no Hamilton $M$-cycle. Recall first the definition of a P{\'o}sa rotation of a path and its pivot (see \cite{POS}). Say $(v_0,...,v_{\ell}) = P\subseteq \Gamma$ is a path. If $\{v_i,v_{\ell}\} \in E(\Gamma )$ then we say that the path $(v_0,...,v_{i-1},v_i,v_{\ell},v_{\ell -1},...,v_{i+1})$ is obtained from $P$ by a rotation with fixed point $v_0$ and pivot $v_i$.

Say that a rotation of an $M$-path $P$ is $M$\emph{-respecting} if the pivot is not a vertex in $V(M)$, and observe that if $P'$ is obtained from $P$ by an $M$-respecting rotation then $P'$ is also an $M$-path. Indeed, no new vertices were added to the path in the process, and the unique removed edge cannot be an edge of $M$, since it is incident to the pivot vertex, which is not in $V(M)$. For a path $P$ with $v$ one of its endpoints, denote by $\enm (P,v)$ the set of all endpoints (other than $v$) of paths that can be obtained from $P$ by a sequence of $M$-respecting rotations with fixed point $v$.

We now claim that if $P=(v_0,...,v_{\ell})$ is a maximal $M$-path in $\Gamma$ then
$$
N_{\Gamma}(\enm (P,v_0)) \setminus V(M) \subseteq N_{P}(\enm (P,v_0)) .
$$
Indeed, assume otherwise. Since $P$ is assumed to be a maximal $M$-path, it must be that $N_{\Gamma}(\enm (P,v_0)) \setminus V(M) \subseteq P$. Then there is  $v_i\in N_{\Gamma}(\enm (P,v_0)) \setminus V(M)$, and $u\in \enm (P,v_i)$ such that $\{u,v_i\}\in E(\Gamma )$ and such that $v_{i-1},v_{i+1} \notin \enm (P,v_0)$. Let $P'$ be an $M$-path from $v_0$ to $u$ that can be obtained from $P$ by a sequence of $M$-respecting rotations, and $w$ be the successor of $v_i$ on $P'$. Observe that, since $v_{i-1},v_i,v_{i+1} \notin \enm (P,v_0)$, it must be that $w\in \{v_{i-1},v_{i+1}\}$. But since $v_i\notin V(M)$ and in particular $\{v_i,w\}\notin E(M)$, a rotation of $P'$ with fixed point $v_0$ and pivot $v_i$ is $M$-respecting, implying $w\in \enm (P,v_0)$, a contradiction.

Now, since we assumed that $\Gamma$ is $M$-expanding and $|N_{P}(U)|<2|U|$ for every set $U\subset V(P)$ that includes an endpoint of $P$, this implies that if $P=(v_0,...,v_{\ell})$ is maximal then $|\enm (P,v_0)| > n/4$.

To complete the proof, for $u\in \enm (P,v)$ denote by $P_u$ an $M$-path from $v$ to $u$ obtained by $M$-respecting rotations. We claim that if $P$ is maximal and $v$ is an endpoint of $P$ then any edge of $\binom{V(G^*)}{2}$ between $u\in \enm (P,v)$ and $\enm (P_u,u)$ is an $M$-booster. Indeed, if $e$ is such an edge then $P_u\cup \{e\}$ is a cycle of length $|V(P)|$. If $P$ was a Hamilton $M$-path then we obtained a Hamilton $M$-cycle. Otherwise, since $M$-expansion implies connectivity, there is an outgoing edge $\{x,y\}\in e(\Gamma )$, where $x\in V(P)$ and $y\notin V(P)$. Since $M$ is a matching, at least one of the edges incident to $x$ in $P_u$ is not in $E(M)$, which implies the existence of a strictly longer path in $\Gamma \cup \{e\}$. Since there are at least $\frac{1}{2}\cdot \left( \frac{n}{4} \right) ^2 = \frac{1}{32}n^2$ such edges, this finishes the proof.

\end{proof}

\begin{lemma} \label{lemma:exist-a-booster}
With high probability, for every $M$-expander $\Gamma \subseteq G^*$ that is not $M$-Hamiltonian and has at most $\frac{c}{900}n$ edges, $E(G^*)$ contains an $M$-booster with respect to $\Gamma$.
\end{lemma}

\begin{proof}
By Lemma \ref{lemma:many-boosters}, $\binom{V(G^*)}{2}$ contains at least $\frac{1}{32}n^2$ $M$-boosters with respect to any such expander. We use the union bound over all choices for the triple $\Gamma , V(G^*), V(M)$ to bound the probability that $G^*$ contains such an expander but none of its $M$-boosters, here using the fact that $V(M) \subseteq V(G^*) \subseteq V(G)$, that is, the pair $V(M),V(G^*)$ defines a partition of $V(G)$ into three sets, and therefore there are at most $3^n$ choices for the pair $V(M),V(G^*)$. We get

$$
 \sum _{k=n}^{\frac{c}{900}n} 3^n \cdot \binom{\binom{n}{2}}{k}\cdot p^k \cdot (1-p)^{\frac{1}{32}n^2} 
\le  \sum _{k=n}^{\frac{c}{900}n} 3^n \cdot \left( \frac{en^2p}{2k} \right) ^k \cdot e^{-\frac{c}{32}n} 
\le  cn\cdot 3^n \cdot \left( 450e \right) ^{\frac{c}{900}n} \cdot e^{-\frac{c}{32}n} 
= o(1),
$$
where in the last inequality we used the fact that $\left( \frac{en^2p}{2k} \right) ^k$ is increasing when $k\le \frac{c}{900}n$, and therefore if $k\le \frac{c}{900}n$ then this size is at most $\left( 450e \right) ^{\frac{c}{900}n}$.

\end{proof}

\begin{lemma} \label{lemma:gamma}
With high probability $G^*$ contains an $M$-expander $\Gamma _0$ with at most $\frac{c}{950}n$ edges.
\end{lemma}

\begin{proof}
We describe a construction of a random subgraph $\Gamma _0$ of $G^*$ with at most $\frac{c}{950}n$ edges and prove that if $G$ has properties \ref{item:V1-is-small}-\ref{item:not-too-few-edges} then the constructed subgraph is an $M$-expander with positive probability, which implies existence.

Construct a random subgraph $\Gamma _0$ of $G^*$ as follows. For every $v\in V(G^*)$ set $E_v$ to be $E_{G^*}(v,V(G^*)\setminus V(M))$ in the case $v\in \sm $, and otherwise set $E_v$ to be a subset of $E_{G^*}(v,V(G^*)\setminus V(M))$ of size $\frac{c}{1000}$, chosen uniformly at random. The random subgraph $\Gamma_0$ is the $G^*$-subgraph with edge set $M\cup \bigcup _{v\in V(G^*)}E_v$. Observe that, since the minimum degree of a vertex into $V(G^*)\setminus V(M)$ in $G^*$ is at least 2, this is also true for a graph $\Gamma_0$ drawn this way, and that $e(\Gamma_0 ) \le \frac{c}{1000}n+\frac{1}{2}n\le \frac{c}{950}n$.

We bound from above the probability that $\Gamma_0$ contains a subset $U$ with at most $n/4$ vertices with less than $2|U|$ neighbours in $V(G^*)\setminus V(M)$. Let $|U|=k\le \frac{n}{4}$, and denote $U_1 = U\cap \sm ,U_2 = U\setminus U_1$ and $k_1,k_2$ the sizes of $U_1,U_2$ respectively. Observe that $V(\Gamma_0 ) \cap \cl = \emptyset$ implies that \emph{(i)} every vertex in $U_2$ has at most one neighbour in $U_1 \cup N_G(U_1) \cup V(M) \subseteq \sm \cup N_G(\sm )$, and therefore $|N_{\Gamma_0}(U_2) \cap (U_1 \cup N_{\Gamma_0}(U_1) \cup V(M))| \le k_2$; and \emph{(ii)} distinct vertices in $\sm (G)$ have non-intersecting neighbourhoods, and therefore $|N_{\Gamma_0}(U_1) \setminus V(M)| \ge 2k_1$.

First we show that if $k_2 \le 2\cdot 10^{-6}n$ then $|N_{\Gamma_0}(U) \setminus V(M)|\ge 2|U|$ with probability 1. Indeed, in this case $|N_{\Gamma_0}(U_2)\setminus V(M)| \ge 4k_2$. Assume otherwise, then $U_2 \cup (N_{\Gamma_0}(U_2)\setminus V(M))$ is a set of size at most $5k_2 \le 10^{-5}n$ which spans at least $\frac{c}{1000}\cdot k_2 - e(U_2)$ edges. But by \ref{item:not-too-many-edges}, $U_2$ spans at most $10^{-4}ck_2$ edges, and $U_2 \cup (N_{\Gamma_0}(U_2)\setminus V(M))$ spans at most $5\cdot 10^{-4}ck_2$ edges, so we get $$10^{-3}c\cdot k_2 - 10^{-4}ck_2 \le 5\cdot 10^{-4}ck_2,$$ a contradiction. Now
\begin{eqnarray*}
|N_{\Gamma_0}(U) \setminus V(M)| & \ge & |N_{\Gamma_0}(U_1)\setminus (U_2\cup V(M))| + |N_{\Gamma_0}(U_2)\setminus (N_{\Gamma_0}(U_1) \cup U_1 \cup V(M)) | \\
& \ge & 2 k_1 - k_2 + 4k_2 - k_2 \\
& \ge & 2k_1+2k_2=2|U|.
\end{eqnarray*}

Now assume that $2\cdot 10^{-6}n\le k_2 \le \frac{1}{4}n$. We show that $|N_{\Gamma_0}(U) \setminus V(M)| \ge 2|U|$ for every $U\subseteq V(G^*)$ with $|U|\le \frac{n}{4}$ with positive probability. Suppose otherwise, then $|V(G^*)\setminus (U \cup N_{\Gamma_0}(U) \cup V(M))| \ge \frac{1}{5}n$. In particular, by \ref{item:BIG-is-small} there are disjoint sets $U'\subseteq U_2 \setminus \lrg $ and $W\subseteq V(G^*)\setminus (U \cup N_{\Gamma_0}(U) \cup V(M))$, of sizes $10^{-6}n$ and $\frac{1}{5}n$ respectively, such that $e_{\Gamma_0}(U',W)=0$. Observe that by \ref{item:not-too-few-edges}, $e_{G^*}(U',W) \ge e_{G}(U',W) \ge 10^{-7}cn$. For a given pair of subsets $U',W$, by Lemma \ref{binom-coeff}(\ref{binom-coeff-2}) the probability for this is at most

\begin{eqnarray*}
\prod_{u\in U'}\pr (e_{\Gamma_0}(u,W)=0) & \le & \prod _{u \in U'} \frac{\binom{e_{G^*}(u,V(G^*)\setminus V(M)) - e_{G^*}(u,W)}{10^{-3}c}}{\binom{e_{G^*}(u,V(G^*)\setminus V(M))}{10^{-3}c}} \\
& \le & \prod _{u \in U'} e^{-\frac{c}{1000}\cdot \frac{e_{G^*}(u,W)}{d_{G^*}(u)}}\\
& \le & \exp \left( {-\frac{c}{1000}\cdot \frac{e_{G^*}(U',W)}{20c}} \right) \\
& \le & \exp \left( -10^{-12}\cdot cn \right)
\end{eqnarray*}

Since there are at most $3^n$ pairs of subsets $U',W$, by the union bound the probability that two subsets of this size with no edges between them in $\Gamma_0$ exist is at most $3^n\cdot \exp \left( -10^{-12}\cdot cn \right) = o(1)$, for large enough $c$. Consequently, the random subgraph $\Gamma_0$ is an $M$-expander with probability $1-o(1)$, implying that $G^*$ contains a sparse $M$-expander, as claimed.

\end{proof}

\begin{lemma}
With high probability $G^*$ is $M$-Hamiltonian.
\end{lemma}

\begin{proof}
This is an immediate consequence of Lemmas \ref{lemma:exist-a-booster}, \ref{lemma:gamma}. By Lemma \ref{lemma:gamma}, with high probability, $G^*$ contains an $M$-expander $\Gamma _0\subseteq G^*$ with at most $\frac{c}{950}n$ edges. By Lemma \ref{lemma:exist-a-booster}, with high probability, if $\Gamma _0$ is not $M$-Hamiltonian then $E(G^*)$ contains an $M$-booster $e_1$ with respect to $\Gamma _0$, an $M$-booster $e_2$ with respect to $\Gamma _0 \cup \{e_1\}$ and so on. After adding at most $n$ such $M$-boosters, the resulting subgraph is already $M$-Hamiltonian, and consequently so is $G^*$.
\end{proof}

To finish our proof, we now claim the $G^*$ being $M$-Hamiltonian with high probability implies Theorem \ref{main}. Indeed, if $C$ is a Hamilton $M$-cycle in $G^*$, then since $E(G^*)\subseteq E(G\cup M)$ it is an $M$-cycle in $G\cup M$ of length $|V(G^*)|$. Therefore, the cycle $C'$ obtained by replacing every edge $\{u,v\}$ of $M$ with the 3-path $(u,u',v',v)$ is contained in $G\cup M'$, and has length
$$
|V(C')| = |V(C)| + |V(M')| = |V(G^*)| + |V(M)|,
$$
which, by Lemmas \ref{lemma:Gstar-is-big} and \ref{lemma:M-is-big}, is at least
$$
(1-(1+\varepsilon /4)\cdot ce^{-c})\cdot n + (1-\varepsilon /4)\cdot ce^{-c}n = \left( 1-\frac{1}{2}\varepsilon ce^{-c} \right) \cdot n  .
$$

As discussed earlier, by removing $M'$ from this cycle we are left with at most $(1+\frac{1}{2}\varepsilon )\frac{1}{2}ce^{-c}n$ vertex disjoint paths that cover at least $\left( 1-\frac{1}{2}\varepsilon ce^{-c} \right) \cdot n$ vertices, and therefore by using additional $\frac{1}{2}\varepsilon ce^{-c} n$ paths of length 0 to cover the remaining vertices we get a disjoint path covering of $G$ with at most $(1+\varepsilon )\frac{1}{2}ce^{-c}n$ paths, thus finishing the proof. \hfill $\square$

\paragraph*{Acknowledgements.} We would like to thank the anonymous referees for this paper for their careful reading and helpful remarks.

\end{document}